\documentclass[fleqn,a4paper,11pt]{article}
\usepackage{amsmath,amsfonts,calrsfs,amssymb,color}
\usepackage{amsthm}

\newtheorem{Theorem}{Theorem}[section] 
\newtheorem{Definition}[Theorem]{Definition}
\newtheorem{Proposition}[Theorem]{Proposition}
\newtheorem{Lemma}[Theorem]{Lemma}
\newtheorem{Corollary}[Theorem]{Corollary}
\newtheorem{Remark}[Theorem]{Remark}

\makeatletter
\@addtoreset{equation}{section}

\makeatother

\def\R{\mathbb R}
\def\N{\mathbb N}

\def\E{\mathbb E}
\def\P{\mathbb P}
\def\ds{\displaystyle}

\title{An integral inequality for  the invariant measure  of a stochastic reaction--diffusion equation}

\author{Giuseppe Da Prato
\\\normalsize  Scuola Normale Superiore di Pisa 
\and
Arnaud Debussche\\
\normalsize IRMAR and \'Ecole Normale Sup\'erieure de Rennes\\\\
Dedicated to Jan Pr\"uss}
 
\date{}

\begin{document}

\maketitle

\begin{abstract}
We consider a reaction--diffusion equation  perturbed by noise (not necessarily white). We prove an integral inequality for  the invariant measure $\nu$ of a stochastic reaction--diffusion equation. Then we discuss some consequences as an integration by parts formula which extends to $\nu$ a basic identity of the Malliavin Calculus.
Finally, we prove the existence of a surface measure for a ball and a half-space of $H$.

\end{abstract}

\noindent {\bf 2010 Mathematics Subject Classification AMS}: 60H15, 35K57, 28C20\medskip


\section{Introduction and preliminaries}
  In the  recent paper \cite{DaDe14}  the following  inequality   involving the invariant measure $\nu$  of the Burgers  equation    was proved,
 \begin{equation}
\label{e5.1}
\left|\int_H \langle    (-A)^{-\alpha/2}D\varphi,z\rangle\,d\nu   \right|\le C_p\|\varphi\|_{L^p(H,\nu)}\,|z|,\quad \alpha>1,
\end{equation}
 for all $\varphi\in C^1_b(H)$,  $z\in H$ and all $p>1$. Here  $A$ is  the Laplace operator equipped with Dirichlet boundary conditions,  $\alpha>1$ and $D$ represents the gradient. As noticed in that paper inequality  \eqref{e5.1} implies  that  $(-A)^{-\alpha/2}D$ is closable in $L^p(H,\nu)$ for all $p\ge 1$. Moreover, for each $z\in H$ there exists $v_z\in L^p(H,\nu)$ for all $p>1$ such that
\begin{equation}
\label{e5.2}
\int_H \langle    (-A)^{-\alpha/2}D\varphi,z\rangle\,d\nu   =\int_H v_z\, \varphi\,d\nu,\quad\forall\;\varphi\in C^1_b(H).
\end{equation}
 We recall that if $\nu=N_Q$  (the Gaussian measure of mean  $0$ and covariance $Q$),  identity \eqref{e5.2} with  $A=-\frac12\,Q^{-1}$ and $\alpha=1$ is well known in Malliavin Calculus. In this case the adjoint $(Q^{1/2}D)^*$ of $Q^{1/2}D$ is  called the Skorhood operator.
 Moreover, \eqref{e5.2}  implies   that $\nu$ is Fomin differentiable in all directions of  $D((-A)^{\alpha/2})$ for $\alpha>1$.  For the definition  of Fomin differentiability see e. g. \cite[Definition 1]{Pu98}.\medskip

The aim of the present paper is to extend inequality \eqref{e5.1} to the following reaction--diffusion equation in $H=L^2(\mathcal O)$ where $\mathcal O$  is a bounded domain of $\R^n$ with sufficiently regular boundary.
\begin{equation}
\label{e4.1}
\left\{\begin{array}{lll}
dX(t)=[AX(t)+p(X(t))]dt+(-A)^{-\gamma/2}dW(t),\\
\\
X(0)=x.
\end{array}\right.
\end{equation}
where  $A$ is  the realization of the Laplace operator  $\Delta_\xi $ equipped with
Dirichlet boundary conditions,
$$
Ax=\Delta_\xi x,\quad x\in D(A),\quad D(A)=H^2(\mathcal O)\cap H^1_0(\mathcal O),
$$
 and    $\frac{n}2-1<\gamma<1$ (which  
obviously  implies that $n\le 3$). Moreover, $p$ is a decreasing polynomial   
of odd degree equal to $N>1$ and  $W$  is an $H$--valued cylindrical Wiener process  on a filtered
probability space
$(\Omega,\mathcal F, (\mathcal F_t)_{t>0},\P)$. 

 It is well known that   equation \eqref{e4.1} has a unique strong solution and that the associate transition semigroup
$$
P_t\varphi(x)=\E[\varphi(X(t,x))],\quad \varphi\in B_b(H)
$$
 possesses a unique invariant measure $\nu$ such that
 \begin{equation}
 \label{e1.4g}
 M_m:=\int_H|x|^m\,\nu(dx)<\infty,\quad\forall\,m\in\N,
 \end{equation}
  see e.g. \cite{Da04}.

 If $\gamma <\frac{n}2-1$, equation \eqref{e4.1} is not expected to have solutions with positive 
 spatial regularity and the equation has to be renormalized. This has been studied in \cite{DaDe-03} for $n=2$ and more recently in \cite{Hai14} and \cite{CaCh14} for $n=3$. 
 \begin{Remark}
 \em
 All following results remain true replacing  $(-A)^{-\gamma/2}$ with  $B=G(-A)^{-\gamma/2}$ and $G\in L(H)$.
  Also, the assumption   $p$   decreasing   could be replaced by $p'$  bounded above.
  
  \end{Remark}
  The main result of the paper is the following
    \begin{Theorem}
  \label{t1}
  Let  $\delta\in (0,1-\gamma),\;p\in (1,\infty)$.Then
   there exists  $C_p>0$ such that  for all       $\varphi\in L^p(H,\nu)$   we have
 $$
\int _H \langle D\varphi(x),h\rangle  \, \nu (dx) \le C_p\|\varphi\|_{L^p(H,\nu)} \,|h|_{1+\delta+\gamma},\quad \forall\;h\in H^{ 1+\delta+\gamma}(\mathcal O).
$$

 \end{Theorem}
 \begin{Remark}
 \em Theorem \ref{t1} has  already been proved in the paper \cite{DaDe15} in  the particular case when $\delta=1-\gamma$.
 As we shall see, the general case  requires new tools as the estimates of $DX(t,x)$ in  some Sobolev spaces, to which is devoted Section 2 below.
 
 \end{Remark}

 Let us describe the content of the paper. In Section 3 we  use the estimates obtained in Section 2 to prove Theorem \ref{t1} and we discuss some consequences proving moreover a   basic integration by parts formula, see \eqref{e3.14e} below. Finally, Section 4 is devoted to discuss surface integrals with respect to the measure $\nu$, in particular for balls and half--spaces. \bigskip

 We conclude this section with some notation and preliminary.
   The norm of $H=L^2(\mathcal O)$ will be  denoted by $|\cdot|$ and the inner product by $ \langle\cdot,\cdot   \rangle$. For $p\ge 1$, $|\cdot|_{L^p}$ is  the norm of $L^p(\mathcal O)$. 
The operator $A$  is self--adjoint negative and there exist an orthonormal basis $(e_h)$ in $H$ and an increasing  sequence of positive numbers
$(\alpha_h)$ such that
\begin{equation}
\label{e1.8}
Ae_h=-\alpha_h e_h,\quad\forall\;h\in\N.
\end{equation}
For any $\alpha\in\R$, $(-A)^\alpha$ denotes the 
$\alpha$ power of the operator $-A$ and   $|\cdot|_\alpha$ is the norm of $D((-A)^{\alpha/2})$ which is equivalent to the norm of the Sobolev space $H^\alpha(\mathcal O)$. We have
$|\cdot|_0=|\cdot|=|\cdot|_{L^2}$. We shall use the interpolatory estimate
\begin{equation}
\label{e1.10}
|x|_b\le |x|^{\frac{c-b}{c-a}}_a\;|x|^{\frac{b-a}{c-a}}_c,\quad -\infty<a<b<c<+\infty.
\end{equation}
Let $\alpha\in (0,1)$.
By the Sobolev embedding theorem we have
$H^{ \alpha}(\mathcal O)\subset L^q(\mathcal O),$ where $q=\tfrac{2n}{n-\alpha}.$
By duality it follows that $L^{\frac{2n}{n+\alpha}}(\mathcal O)\subset H^{-\alpha}(\mathcal O)$
and therefore
\begin{equation}
\label{e1.12}
|x|_{-\alpha}\le c|x|_{\frac{2n}{n+\alpha}},\quad \forall\,\alpha\in(0,1),\,x\in L^{\frac{2n}{n+\alpha}}(\mathcal O) ,
\end{equation}
for a suitable constant $c>0$.\medskip

Moreover, it is convenient to introduce  the following approximating problem
\begin{equation}
\label{e4.9}
\left\{\begin{array}{l}
dX_\epsilon (t)=(AX_\epsilon (t)+p_\epsilon (X_\epsilon (t))dt+(-A)^{-\gamma/2}dW(t),\\
\\
X_\epsilon (0)=x\in H,
\end{array}\right.
\end{equation}
where for any $\epsilon >0,$   $p_\epsilon$  are the Yosida approximations of $p$, that is
$$
p_{\epsilon}(r)= \frac{1}{\epsilon}
\;(r-J_{\epsilon}(r)),\;J_{\epsilon}(r)=(1-\epsilon p(\cdot))^{-1}(r),\quad r\in \R.
$$ 
Let us  denote by $P^\epsilon_t$  the approximate transition semigroup 
$$
P_t^\epsilon\varphi(x)=\E[\varphi(X_\epsilon(t,x))],\quad \varphi\in B_b(H)
$$
 and recall   the following  Bismut-Elworthy-Li formula, see \cite{Ce01}.
 \begin{equation}
\label{e4.10}
 \langle DP^\epsilon _t\varphi(x),h  \rangle=\frac1t\;\E\left[\varphi(X_\epsilon (t,x))\int_0^t\langle 
(-A)^{\frac{\gamma}{2}}\eta_\epsilon^h(s,x),dW(s) \rangle  \right],\quad h\in H, 
\end{equation}
 where  $\eta_\epsilon^h(t,x)=:DX_\epsilon(t,x)\cdot h$  is the $x$--derivative of $X_\epsilon(t,x)$   in the direction $h$. As well known, $\eta_\epsilon^h(t,x)$ is the  solution of the equation
 \begin{equation}
\label{e5b}
\left\{\begin{array}{lll}
 \ds\frac{d}{dt}\;\eta_\epsilon ^{h}(t,x)&=&A\eta_\epsilon ^{h}(t,x)-p'_\epsilon(X_\epsilon(t,x))
 \eta_\epsilon^{h}(t,x),\\
 \\
  \eta_\epsilon^{h}(0,x)&=&h.
  \end{array}\right.
\end{equation}
Let us recall the following elementary but useful identity (proved   in   \cite{DaDe15}).
 For any  $\varphi\in C^1_b(H)$,  $\epsilon>0$, $h,x\in D(A)$,  we have
  \begin{equation}
\label{e7}
\begin{array}{l}
  \ds P^\epsilon_t(\langle D\varphi,h\rangle)= \langle DP^\epsilon_t\varphi, h\rangle 
  \ds-   \int_0^tP^\epsilon_{t-s}(\langle Ah+p_\epsilon' h  ,D P^\epsilon_s\varphi\rangle)ds.
  \end{array}
\end{equation}\medskip

  \section{Estimates of $\eta_\epsilon^{h}(t,x)$ and of $DP^\epsilon_t\varphi(x)$}
  
  Let us fix $T>0$, $x\in H$. We are going to state some
estimates of $\eta_\epsilon^{h}(t,x)$ in the norms $|\cdot|_{-\alpha}$ and  $|\cdot|_{1-\alpha}$, needed below. We shall denote by $\Delta_{T,\epsilon}(X)$ the random variable
    \begin{equation}
\label{e2.1f}
\Delta_{T,\epsilon}(X)=\sup_{t\in[0,T]}|p_\epsilon'(X_\epsilon(t,x))|_\infty^2+1  .\end{equation}
 \begin{Lemma}
\label{l2.2e}
    For any $\alpha\in (0,1)$ we have
        \begin{equation}
\label{e2.7e}
\int_0^T|\eta_\epsilon^h(t,x)|^2_{1-\alpha}\,dt\le C(T)\Delta_{T,\epsilon}(X)\, |h|^2_{-\alpha},\quad\forall\;h\in H^{-\alpha}(\mathcal O).
\end{equation}
\end{Lemma}
\begin{proof}
We shall proceed by duality. Let us consider the mapping
$$
U:H\to L^2(0,T;H),\quad h\to \eta^h(\cdot,x),
$$
(here $x\in H$ and $\omega\in\Omega$ are fixed.)
 Let us find the adjoint $U^*$ of $U$.  Write for $h\in H$, $F\in L^2(0,T;H),$
  $$
\begin{array}{l}
 \ds\langle U(h),F   \rangle_{L^2(0,T;H)}=\int_0^T   \langle\eta_\epsilon^h(t,x),F(t)   \rangle\,dt\\
 \\
 \ds=\int_0^T   \langle G(t,0,x)h,F(t)   \rangle\,dt=
 \left<h, \int_0^TG(t,0,x)F(t)\,dt    \right>,
 \end{array} 
$$
where $G(t,s,x)$ is the evolution operator of the family of linear operators \break $(A+p_\epsilon'(X(t,x)))_{t\in[0,T]}$. Noticing the $G(t,s,x)$ is symmetric
we have
\begin{equation}
\label{e2.8e}
U^*(F)=\zeta(0),\quad\forall\; F\in L^2(0,T;H),
\end{equation}
where $\zeta$ is the solution of the problem
\begin{equation}
\label{e2.9e}
\left\{\begin{array}{l}
\zeta'(t)=-A\zeta(t)-p_\epsilon'(X(t,x))\zeta(t)+F(t),\\
\\
\zeta(T)=0.
\end{array}\right. 
\end{equation}

Now \eqref{e2.7e} is equivalent to
\begin{equation}
\label{e2.10e}
\int_0^T|(-A)^{\frac{1-\alpha}{2}}U(h)(s)|^2\,ds\le C(T)\,\Delta_{T,\epsilon}(X) |(-A)^{-\frac{\alpha}{2}}h|^2,
\end{equation}
which  by duality it equivalent to
\begin{equation}
\label{e2.11e}
 |(-A)^{\frac{\alpha}{2}}\zeta(0)|^2\le C(T)\,\Delta_{T,\epsilon}(X)\int_0^T|(-A)^{\frac{\alpha-1}{2}}F(t)|^2\,dt.
 \end{equation}
 To prove \eqref{e2.11e}, we shall proceed in two steps.\medskip

{\it Step 1}. We have
\begin{equation}
\label{e2.12e}
|\zeta(t)|^2\le \int_0^T |F(s)|^2_{-1}ds
\end{equation}

Taking the inner product in \eqref{e2.9e}  with $\zeta$ and recalling that $p_\epsilon'(r)\le 0$, we find
$$
\begin{array}{lll}
\ds\frac12\,\frac{d}{dt}\,|\zeta(t)|^2&=&\ds|\zeta(t)|_1^2-\int_\mathcal Op_\epsilon'(X_\epsilon(t,x))\,\zeta(t)^2\,dx+\int_\mathcal O F(t)\,\zeta(t)\,dx\\
\\
&\ge&\ds  |\zeta(t)|_1^2-|\zeta(t)|_1\,|F(t)|_{-1}\\
\\
&\ge&\ds\frac12\, |\zeta(t)|_1^2 -\frac12\,|F(t)|^2_{-1}.
\end{array}
$$
It follows that
$$
\frac{d}{dt}\, |\zeta(t)|^2 \ge  |F(t)|^2_{-1},\quad\forall\;t\in[0,T].
$$
Integrating from $t$ and $T$, yields  \eqref{e2.12e}. \medskip

{\it Step 2}. Conclusion.\medskip

Now we take the inner product in \eqref{e2.9e}  with $(-A)^\alpha\zeta$,
$$
\begin{array}{lll}
\ds\frac12\,\frac{d}{dt}\,|\zeta(t)|_\alpha^2&=&\ds |\zeta(t)|_{\alpha+1}^2-\int_\mathcal Op_\epsilon'(X_\epsilon(t,x))\,\zeta(t)\,(-A)^\alpha\zeta(t)\,dx\\
\\
&&\ds+\int_\mathcal O F(t)\,(-A)^\alpha\zeta(t)\,dx\\
\\
&\ge&\ds  |\zeta(t)|_{\alpha+1}^2-|p_\epsilon'(X_\epsilon(t,x))|_\infty\,|\zeta(t)|\,|\zeta(t)|_{2\alpha}-|\zeta(t)|_{\alpha+1}\,|F(t)|_{\alpha-1}.
\end{array}
$$
Since $\alpha\le 1$ by the Poincar\'e inequality we have
$$
|\zeta(t)|_{2\alpha}\le c|\zeta(t)|_{\alpha+1}
$$
and consequently
$$
\frac12\,\frac{d}{dt}\,|\zeta(t)|_\alpha^2\ge -c|p_\epsilon'(X(t,x))|_\infty^2 |\zeta(t)|^2-|F(t)|^2_{\alpha-1}.
$$
Integrating from $0$ and $T$, yields
\begin{equation}
\label{e2.13e}
|\zeta(0)|^2_\alpha\le c|p_\epsilon'(X_\epsilon(t,x))|_\infty^2 \int_0^T|\zeta(t)|^2dt+\int_0^T|F(t)|^2_{\alpha-1}\,dt.
\end{equation}
Now by \eqref{e2.12e} we deduce
\begin{equation}
\label{e2.14e}
|\zeta(t)|^2\le  \int_t^T |F(s)|^2_{-1}ds\le C(T)\int_0^T |F(s)|^2_{\alpha-1}ds.
\end{equation}
Substituting in \eqref{e2.13e}, yields
$$
|\zeta(0)|^2_\alpha\le C(T)\,\Delta_{T,\epsilon}(X)\int_0^T|F(s)|^2_{\alpha-1}ds,
$$
as required.
 \end{proof}

\begin{Lemma}
\label{l2.3e}
\begin{equation}
\label{e2.15e}
|\eta_\epsilon^h(t,x)|_{-\alpha}\le  C(T)\,\Delta_{T,\epsilon}(X) |h|^2_{-\alpha},\quad\forall\;h\in H^{-\alpha}(\mathcal O).
\end{equation}
\end{Lemma}
\begin{proof}
Fix $t\in[0,T]$  and introduce a mapping
$$
V_t:H\to H,\quad h\to \eta_\epsilon^h(t,x),
$$
(where $x\in H$ and $\omega\in\Omega$) are fixed.
Then $V_t^*(h)= G(t,0)k$ and so $V_t^*(k)=\zeta(0)$ where $\zeta$ is the solution to
\begin{equation}
\label{e2.16e}
\left\{\begin{array}{l}
\zeta'(t)=-A\zeta(t)-p_\epsilon'(X(t,x))\zeta(t) ,\\
\\
\zeta(t)=k.
\end{array}\right. 
\end{equation}
We have clearly
\begin{equation}
\label{e2.17e}
|\zeta(s)|^2\le \,|k|^2,\quad \forall \;s\in[0,t].
\end{equation}
To prove \eqref{e2.15e}, arguing by duality, it is enough to prove that
\begin{equation}
\label{e2.18e}
|\zeta(t)|^2_{\alpha}\le C(T)\,\Delta_{T,\epsilon}(X)\,|h|^2_{\alpha}.
\end{equation}
 Write
\begin{equation}
\label{e2.19e}
\begin{array}{lll}
\ds\frac12\,\frac{d}{dt}\,|\zeta(t)|_\alpha^2&=&\ds |\zeta(t)|_{\alpha+1}^2-\int_\mathcal Op_\epsilon'(X_\epsilon(t,x))\,\zeta(t)\,(-A)^\alpha\zeta(t)\,dx\\
\\
&\ge&\ds  |\zeta(t)|_{\alpha+1}^2-|p'(X_\epsilon(t,x))|_\infty|\zeta(t)|\,|\zeta(t)|_{-2\alpha} \\
\\
&\ge&\ds -c|p'(X_\epsilon(t,x))|_\infty |\zeta(t)|^2.
\end{array} 
\end{equation}
Therefore, thanks to \eqref{e2.17e},
$$
\begin{array}{lll}
\ds|\zeta(t)|_\alpha^2&\le &\ds |k|_\alpha^2+\int_0^t |p_\epsilon'(X_\epsilon(s,x))|^2_\infty\,|\zeta(s)|^2ds\\
\\
&\le&\ds |k|_\alpha^2+\sup_{t\in[0,T]}|p_\epsilon'(X_\epsilon(s,x))|^2_\infty\, \,t|k| ^2\\
\\
&\le&\ds C(T)\,\Delta_{T,\epsilon}(X)\,)|k|_\alpha^2 
\end{array} 
$$
as required.

\end{proof}
\begin{Corollary}
\label{c2.3e}
Let $\delta\in(0,1-\alpha)$, then
\begin{equation}
\label{e2.20e}
\int_0^T|\eta_\epsilon^h(t,x)|^2_{1-\alpha-\delta}\,dt\le  C(T) T^{\delta}\,\Delta_{T,\epsilon}(X) \, |h|^{2}_{-\alpha}.
\end{equation}

\end{Corollary}
\begin{proof}
By \eqref{e1.10} with $a=-\alpha,\;b=1-\alpha-\delta,\;c=1-\alpha$, we have
\begin{equation}
\label{e2.21e}
\int_0^T|\eta_\epsilon^h(t,x)|^2_{1-\alpha-\delta}\,dt\le \int_0^T|\eta_\epsilon^h(t,x)|^{2 \delta}_{-\alpha}\,|\eta_\epsilon^h(t,x)|^{2(1-\delta)}_{1-\alpha}\,dt.
\end{equation}
Now, taking into account  \eqref{e2.15e}, we have
\begin{equation}
\label{e2.22e}
\begin{array}{l}
\ds\int_0^T|\eta_\epsilon^h(t,x)|^2_{1-\alpha-\delta}\,dt\le  C(T)^{\delta}\,(\Delta_{T,\epsilon}(X ))^{\delta} |h|^{2 \delta}_{-\alpha}\,\int_0^T |\eta_\epsilon^h(t,x)|^{2(1-\delta)}_{1-\alpha}\,dt.
\end{array}
\end{equation}
By H\"older's inequality with exponents $\tfrac1{1-\delta},\,\tfrac1\delta$, we deduce  
\begin{equation}
\label{e2.23e}
\begin{array}{l}
\ds\int_0^T|\eta_\epsilon^h(t,x)|^2_{1-\alpha-\delta}\,dt\le  C(T)^{\delta}(\Delta_{T,\epsilon}(X))^{\delta} |h|^{2\delta}_{-\alpha} T^{\delta}\left(\int_0^T |\eta_\epsilon^h(t,x)|^{2}_{1-\alpha}\,dt\right)^{1-\delta}.
\end{array}
\end{equation}
Finally, taking into account \eqref{e2.7e}, yields
  \begin{equation}
\label{e2.24e}
\int_0^T|\eta_\epsilon^h(t,x)|^2_{1-\alpha-\delta}\,dt\le C(T)\left(\sup_{t\in[0,T]}|p_\epsilon'(X_\epsilon(t,x))|_\infty^2+1\right) |h|^2_{-\alpha},
\end{equation}
and \eqref{e2.20e} follows.

\end{proof}

Now we state an  estimate for $DP^\epsilon_t\varphi(x)$.  Fix $\delta\in (0,1-\gamma)$ and set $\alpha=1-\gamma-\delta$.
\begin{Lemma}
\label{l2.4e}
Let $p>1$,  $\delta\in (0,1-\gamma)$ and set $\alpha=1-\gamma-\delta$. Then we have
\begin{equation}
\label{e2.25e}
  |\langle DP^\epsilon_t\varphi(x),h   \rangle|\le  C(t)t^{\frac\delta2-1}\,\left[ \E\left( \varphi^p(X_\epsilon(t,x))\right)   \right] ^{1/p}
  \,\left[ \E\left(\Delta_{T,\epsilon}(X))^{p'/2}\right)   \right] ^{1/p'}\;|h|_{-\alpha},
\end{equation}
with $\frac1{p}+\frac1{p'}=1$.
\end{Lemma}
\begin{proof}
By \eqref{e4.10} and H\"older's inequality we have
\begin{equation}
\label{e2.26e}
| \langle DP^\epsilon _t\varphi(x),h  \rangle|\le \frac1t\;\left[\E\left(\varphi^p(X_\epsilon (t,x))\right)\right]^{1/p}\; \left[\E\left(\int_0^t\langle 
(-A)^{\frac{\gamma}{2}}\eta_\epsilon^h(s,x),dW(s) \rangle  \right)^{p'}\right]^{1/{p'}},\quad h\in H, 
\end{equation}
Thanks to  Burkholder's inequality  and recalling that $\gamma=1-\alpha-\delta$ it follows that
\begin{equation}
\label{e2.27e}
\begin{array}{lll}
| \langle DP^\epsilon _t\varphi(x),h  \rangle|&\le &\ds\frac3t\;[\E\left(\varphi^p(X_\epsilon (t,x))\right)]^{1/p}\\
\\
&\times&\ds \left\{\E\left[\left(\int_0^t 
|\eta_\epsilon^h(s,x)|^2_{1-\alpha-\delta}\,ds   \right)^{p'/2}\right]\right\}^{1/p'},\quad h\in H. 
\end{array}
\end{equation}
Now \eqref{e2.26e} follows from   Corollary \ref{c2.3e}.

\end{proof}

\section{Integral estimates of $DP^\epsilon_t$}

\begin{Lemma}
\label{l3.1e}
Let $p>1$, $q>\tfrac{p}{p-1}$,  $\delta\in (0,1-\gamma)$ and set $\alpha=1-\gamma-\delta$. Then there exists $C>0$ such that
\begin{equation}
\label{e3.1e}
 \int_H |\langle DP^\epsilon_t\varphi(x),h(x)   \rangle|\nu(dx) \le  Ct^{\frac\delta2-1}\,\,\|\varphi\|_{L^p(H,\nu)}\,\||h|_{-\alpha}\|_{L^q(H,\nu)},\quad\forall\;h\in L^q(H,\nu),
 \varphi\in C^1_b(H).
\end{equation}
\end{Lemma}
\begin{proof}
By \eqref{e2.25e} we deduce
\begin{equation}
\label{e3.2e}
  |\langle DP^\epsilon_t\varphi(x),h(x)   \rangle|\le  C(t)t^{\frac\delta2-1}\,\left[ \E\left( \varphi^p(X_\epsilon(t,x))\right)   \right] ^{1/p}
  \,\left[ \E\left(\Delta_{T,\epsilon}(X)^{p'/2}\right)   \right] ^{1/p'}\;|h(x)|_{-\alpha}.
\end{equation}
Now we integrate \eqref{e3.2e} with respect to $\nu$ over $H$ and use a triple H\"older's inequality with exponents
$p,\;q, r$ such that $\tfrac1{p}+\tfrac1{q}+\tfrac1{r}=1$.
\begin{equation}
\label{e3.3e}
\begin{array}{l}
\ds \int_H\left|\langle DP^\epsilon_t\varphi(x),h(x)   \rangle\right|\,\nu(dx)\le  C(t)t^{\frac\delta2-1}\,\left(\int_H P_t(\varphi^p)\,d\nu  \right) ^{1/p}\\
\\
\ds
  \times\,\left(\int_H\left\{\E\left[(\Delta_{T,\epsilon}(X)^{p'/2}\right]\right\}^{\frac{r}{p'}}\,d\nu\right)^{\frac{1}{r}}\;\left(\int_H|h(x)|^q_{-\alpha}\,d\nu\right)^\frac1q.
\end{array} 
\end{equation}
But, recalling that $p$ is a polynomial of degree $N$, we see that  there exists $C_1>0$ such that
$$
(\Delta_{T,\epsilon}(X)^{p'/2}\le C_1\sup_{t\in[0,T]}|X_\epsilon(t,x))|^{\frac{p'(N-1)}{2}}_\infty+C_1.
$$
Now by \cite[Theorem 4.8(iii)]{Da04}, 
there is $C_2>0$ such that
$$
\E\left[(\Delta_{T,\epsilon}(X)^{p'/2}\right]
\le C_2 |x|^{\frac{p'(N-1)}{2}} +C_2
$$
and so there is $C_3>0$ such that
\begin{equation}
\label{e3.4e}
\left\{\E\left[(\Delta_{T,\epsilon}(X)^{p'/2}\right]\right\}^{\frac{r}{p'}}\
\le C_3 |x|^{\frac{r(N-1)}{2}} +C_3.
\end{equation}
Consequently, thanks to \eqref{e1.4g} we have
\begin{equation}
\label{e3.4ee}
\int_H\left\{\E\left[(\Delta_{T,\epsilon}(X)^{p'/2}\right]\right\}^{\frac{r}{p'}}\,d\nu=C_4<\infty.
\end{equation}
Substituting \eqref{e3.4ee} into \eqref{e3.3e} and taking into account the invariance of $\nu$ yields the conclusion.

\end{proof}
We are now ready to show the main result of the paper.  
  \begin{Theorem}
  \label{t2}
   Let  $\delta\in (0,1-\gamma),\;p\in (1,\infty)$.Then
   there exists  $C_p>0$ such that  for all       $\varphi\in C_b^1(H,\nu)$   we have
 \begin{equation}
\label{e3.5e}
\int _H \langle D\varphi(x),h\rangle  \, \nu (dx) \le C\|\varphi\|_{L^p(H,\nu)} \,|h|_{1+\delta+\gamma},\quad \forall\;h\in H^{1+\delta+\gamma}(\mathcal O).
\end{equation}
 \end{Theorem}
 \begin{proof} 
  Let us integrate   identity  \eqref{e7}  
   with respect to $\nu_\epsilon$ over $H$. Taking into account the invariance of $\nu_\epsilon$, we obtain
\begin{equation}
\label{e3.6e}
\begin{array}{lll}
 \ds \int_H \langle D\varphi(x), h\rangle \nu_\epsilon(dx)&=&\ds \int_H\langle DP^\epsilon_t\varphi(x),h\rangle\,\nu_\epsilon(dx)\\
 \\
&&- \ds  \int_0^t\int_H \langle Ah+p'_\epsilon(x)h  ,DP^\epsilon_s\varphi(x) \rangle ds \,\nu_\epsilon(dx)\\
\\
=:J_1+J_2.
 \end{array}
\end{equation}

  As for $J_1$ we have by \eqref{e3.1e}
  \begin{equation}
\label{e3.7e}
 J_1=\int_H\langle DP^\alpha_t\varphi(x), h\rangle\nu_\epsilon(dx)\le   Ct^{-1+\delta/2}\;\|\varphi\|_{L^p(H,{\nu_\epsilon})} \;|h|_{-\alpha}. 
\end{equation}
As for $J_2$ we have again by \eqref{e3.1e}, taking into account that
$|Ah|_\alpha=|h|_{2-\alpha}=|h|_{1+\delta+\gamma},$
 \begin{equation}
\label{e3.8e}
 \begin{array}{l}
 \ds J_2= \int_0^t\int_H \langle Ah+{p'_\epsilon}(x)h  ,DP^\epsilon_s\varphi(x) \rangle ds \,d\nu_\epsilon\\
 \\
 \ds\le {C\int_0^ts^{-1+\delta/2}\,ds}\,\|\varphi\|_{L^p(H,\nu_\epsilon)}
  \,(|h|_{1+\delta+\gamma}+\||p'_\epsilon\,h|_{-\alpha}\|_{L^q(H,\nu_\epsilon)})
  \end{array}
\end{equation}
On the other hand,  thanks to \eqref{e1.12} and H\"older's inequality, we have
$$
|p'_\epsilon\,h|_{-\alpha}\le c|p'_\epsilon\,h|_{L^{\frac{2n}{n+\alpha}}} \le
c|p'_\epsilon|_{L^{\frac{2n}{\alpha}}}\;| h|_{L^{2}}.$$
 Moreover thanks to \cite[Poposition 4.20]{Da04}, there is $c'>0$ such that
 $$
\||p'_\epsilon\,h|_{-\alpha}\|_{L^q(H,\nu_\epsilon)}\le c\left( \int_H|p'_\epsilon|^q_{L^{\frac{2n}{\alpha}}}\,d\nu_\epsilon  \right)^{1/q}\;|h|\le c' |h|.
 $$
 Consequently
 \begin{equation}
\label{e3.9e}
J_2\le C\int_0^ts^{-1+\delta/2}\,ds\,\|\varphi\|_{L^p(H,\nu_\epsilon)}
  \,(|h|_{1+\delta+\gamma}+c'|h|).
\end{equation}
By \eqref{e3.7e} and \eqref{e3.9e}, setting $t=1$ we find for a suitable constant $C_p$
\begin{equation}
\label{e3.10e}
\int _H \langle D\varphi(x),h\rangle  \, \nu_\epsilon (dx) \le C\|\varphi\|_{L^p(H,\nu_\epsilon)} \,|h|_{1+\delta+\gamma},\quad \forall\;h\in H^{1+\delta+\gamma}(\mathcal O).
\end{equation}  
Letting $\epsilon\to 0$, yields \eqref{e3.5e} (see \cite[Proposition 14]{DaDe15}).

\end{proof}
   
 \subsection{Some consequences of Theorem \ref{t2}. }
 
 In the following we shall assume that the assumptions of Theorem \ref{t2} are fulfilled and set
 $$
 \beta=\frac{1+\gamma+\delta}{2}.
 $$
 (Notice that $\beta>\tfrac12$.) We shall write \eqref{e3.5e} as
  \begin{equation}
\label{e3.11g}
\int _H \langle (-A)^{-\beta}D\varphi(x),h\rangle  \, \nu (dx) \le C\|\varphi\|_{L^p(H,\nu)} \,|h| ,\quad \forall\;\varphi\in C_b^1(H),\,h\in H.
\end{equation}

The proof of the following result is  similar to     \cite[Proposition 11]{DaDe15}, so, it will be omitted.
\begin{Proposition}
\label{p3.3e}
  The linear operator
 $$
\varphi\in C^1_b(H)\subset L^p(H,\nu)\mapsto (-A)^{-\beta} D\varphi\in L^p(H,\nu;H),$$
is closable in $L^p(H,\nu)$ for all $p>1$.
\end{Proposition}
We shall denote by $M_p$ the closure of  $(-A)^{-\beta} D$ and by  $M^*_p$ the dual of $M_p$, so that 
we have
\begin{equation}
\label{e3.11e}
 \int_H \langle M_p\varphi,F\rangle\,d\nu =\int_H \varphi\,M_p^*(F)\,d\nu,\quad\forall\;\varphi\in D(M_p),\;F\in D(M^*_p).
\end{equation}
   Clearly for any $p\in(0,1),$
$$
M_p:D(M_p)\subset L^p(H,\nu)\to L^p(H,\nu;H)
$$ 
and
$$
M_p^*:D(M_p^*)\subset L^q(H,\nu;H)\to L^q(H,\nu),
$$
where $q=\tfrac{p}{p-1}$.

When no confusion may arise we shall drop the sub-index $p$.
\begin{Remark}
\label{r3.4e}
\em
 Obviously, $F\in D(M_p^*)$ iff there exists $K>0$ such that
 \begin{equation}
\label{e3.12e}
 \left|\int_H\langle (-A)^\beta D\varphi, F\rangle\,d\nu\right|\le K\|\varphi\|_{L^p(H,\nu)},
 \quad\forall\;\varphi\in C^1_b(H).
\end{equation}
In this case we have
\begin{equation}
\label{e3.13e}
\|M_p^*(F)\|_{L^q(H,\nu;H)}\le K.
\end{equation}
\end{Remark}
 Let us show an integration by parts formula.
 \begin{Proposition}
 \label{p3.4e}
For any  $z\in H$   there exists a function  $v_z\in L^q(H,\nu)$ for all $q\in (1,\infty)$ such that
\begin{equation}
\label{e3.14e}
\int_H \langle (-A)^{-\beta} D\varphi(x),z\rangle\,\nu(dx)=\int_H \varphi(x)\,v_z(x)\,\nu(dx)
\end{equation}
and
\begin{equation}
\label{e3.15e}
\|v_z\|_{L^q(H,\nu)}\le C_p|z|,
\end{equation}
where $p=\tfrac{q}{q-1}.$

 Therefore    there exists  the   Fomin derivative of   $\nu$ in all  directions of   $D((-A)^{\beta})$. 
\end{Proposition}
\begin{proof}
Set $F_z(x)=h$ for all $x\in H$. By \eqref{e3.11g} we deduce
 \begin{equation}
\label{e3.17e}
\left|\int _H \langle (-A)^{-\beta}D\varphi(x),F_z(x)\rangle  \, \nu (dx)\right| \le C\|\varphi\|_{L^p(H,\nu)} \,|z| ,\quad \forall\;\varphi\in C^1_b(H),\,z\in H.
\end{equation}
By Remark \ref{r3.4e} it follows that $F_z\in D(M^*_p)$, so that setting $v_z=M^*_p(F_z)$ we see that \eqref{e3.15e} holds.
\end{proof}
Now, recalling \eqref{e1.8} we find
\begin{Corollary}
 \label{c3.6e}
 For all $h\in \N$ we have
 \begin{equation}
\label{e3.18e}
\int_H  D_h\varphi(x) \,\nu(dx)=\alpha_h^\beta\int_H \varphi(x)\,v_h(x)\,\nu(dx)
\end{equation}
where $v_h=v_{e_h}$ and
\begin{equation}
\label{e3.19e}
\|v_h\|_{L^q(H,\nu)}\le C_p,
\end{equation}
where $p=\tfrac{q}{q-1}.$
 \end{Corollary}
Let us now find an expression of $M^*$.
\begin{Proposition}
\label{p3.7e}
Let $F=\sum_{h=1}^m f_h e_h$ with $f_h\in C^1_b(H)$ and $m\in\N$. Then $F\in D(M^*)$ and it results
\begin{equation}
\label{e3.20e}
M^*(F)=-\mbox{\rm div}\,[(-A)^{-\beta}F]+\sum_{h=1}^m f_h v_h.
\end{equation}

\end{Proposition}
\begin{proof}
Write
$$
\begin{array}{l}
\ds \int _H \langle (-A)^{-\beta}D\varphi,F\rangle  \, d\nu =\sum_{h=1}^m \alpha_h^{-\beta}\int_H D_h\varphi\,f_h\,d\nu\\
\\
\ds=\sum_{h=1}^m \alpha_h^{-\beta}\int_H D_h(\varphi\,f_h)\,d\nu-\sum_{h=1}^m \alpha_h^{-\beta}\int_H \varphi\,D_hf_h\,d\nu
\end{array} 
$$
Thanks to Corollary \ref{c3.6e} we have
$$
\int _H \langle (-A)^{-\beta}D\varphi,F\rangle  \, d\nu =\sum_{h=1}^m  \int_H \varphi\,v_h\,f_h\,d\nu- \int_H \varphi\,\mbox{\rm div}\,[(-A)^{-\beta}F]\,d\nu
$$
which implies \eqref{e3.20e}.
\end{proof}

\section{Application to surface integrals}

 We still assume in this section  that the assumptions of Theorem \ref{t2} are fulfilled and set  $\beta=\tfrac{1+\gamma+\delta}{2}.$

Let moreover  $g:H\to \R$  be a Borel function. For all  $\varphi\in L^1(H,\nu)$ we set 
$$
G_\varphi(r)=\int_{\{g\le r\}}\varphi\,d\nu,\quad r\in g(H).
$$
\begin{Definition}
Assume that  $G_\varphi(\cdot)$ is $C^1$ for any  $\varphi\in C_b(H)$ and for $r\in g(H)$  there exists  a measure  $\sigma_r$ concentrated on  $\{g=r\}$ such that
 $$
D_rG_\varphi(r)=\lim_{\epsilon\to 0}\frac1{2\epsilon}\int_{\{r-\epsilon\le  g\le r+\epsilon\}}\varphi\,d\nu=\int_{\{g=r\}}\varphi\,d\sigma_r.
$$
Then  $\sigma_r$ is called the  {\em surface measure} on  $\{g=r\}$ determined by   $\nu$.
\end{Definition}
When $\nu$ is Gaussian several papers have been devoted to prove existence of surface measures, starting from the pioneering  works by   Airault and Malliavin, \cite{AiMa88}  and  Fejer-de La Pradelle, \cite{FePr92}; we also quote \cite{DaLuTu14}.  The typical assumption is
\begin{equation}
\label{e4.1ef}
F_g:=\frac{Mg}{|Mg|^2}\in D(M^*).
\end{equation}
Concerning the measure $\nu$, the following result has been proved in \cite{DaLuTu} as a simple generalization of \cite{DaLuTu14}.
 \begin{Theorem}
\label{t4.1e}
Let $M$ be the closure of  $(-A)^{-\beta}D$.
Assume that   $g\in D(M)$ and
\begin{equation}
\label{e4.1e}
F_g:=\frac{Mg}{|Mg|^2}\in D(M^*).
\end{equation}
Then there exists   a {\em surface measure} $\sigma_r$  on $\{g=r\}$ determined by   $\nu$.
\end{Theorem}
\begin{Remark}
\em A similar result was obtained by  \cite{Pu98} and \cite{BoMa}, under the assumption that the Fomin derivative of $\nu$ exists in sufficiently many directions $z\in H$ where identity \eqref{e3.18e} is fulfilled.

\end{Remark}
Notice that  by \eqref{e3.20e} it follows that (formally) 
\begin{equation}
\label{e4.2e}
\begin{array}{l}
M^*\left(\frac{(-A)^{-\beta}Dg}{|(-A)^{-\beta}Dg|^2}\right)\\
\\
\ds= 
\mbox{\rm div}\;\left[\frac{(-A)^{-2\beta}Dg}{|(-A)^{-\beta}Dg|^2}\right]+\frac1{|(-A)^{-\beta}Dg|^2}\sum_{h=1} ^\infty \alpha_h^{-\beta} v_h D_hg.
\end{array}
\end{equation}

In this section we shall prove that condition \eqref{e4.1e} is fulfilled for  balls and half--spaces.

\subsection{Surface measure of balls}
In this subsection  we shall  assume $n=1$ so that
\begin{equation}
\label{e4.3e}
\sum_{h=1}^\infty \alpha^\beta_h<\infty,
\end{equation}
because $\beta>1/2$.
Let $g(x)=|x|^2$ and take  $r>0$, so that   the surface $\{g=r\}$ coincides with the ball $B_{\sqrt r}$ in $H$ with center $0$ and radius $\sqrt r$. To apply Theorem \ref{t4.1e}
we have to show that $F_g\in D(M^*)$,  where
$$
F_g(x)=\frac{(-A)^{-\beta} x}{2|(-A)^{-\beta} x|^2}.
$$
 To this aim it is convenient to introduce a regular  finite dimensional approximation of $F_g$  setting
$$
F^n_g(x)=\frac{(-A)^{-\beta} P_nx}{n^{-1}+2|(-A)^{-\beta} x|^2},
$$
where $P_n$ is the orthogonal projector  on span $(e_1,...,e_n)$.
Clearly $F^n_g\in D(M^*)$ and by \eqref{e3.20e}  we have  
\begin{equation}
\label{e3.17}
M^*(F_g^n)(x)=-\mbox{\rm div}\;\left[ \frac{(-A)^{-2\beta} P_nx}{n^{-1}+2|(-A)^{-1} x|^2}  \right]+\frac{\sum_{h=1}^n\alpha_h^{-\beta}x_hv_h}{n^{-1}+2|(-A)^{-\beta} x|^2}.
\end{equation}
But for $j=1,...,n$ we have
$$
D_j\left[ \frac{(-A)^{-2\beta} x_j}{n^{-1}+2|(-A)^{-\beta} x|^2}  \right]=\frac{\alpha_j^{-2\beta} }{n^{-1}+2|(-A)^{-\beta} x|^2}-\frac{4\alpha_j^{-4\beta} x^2_j}{(n^{-1}+2|(-A)^{-\beta} x|^2)^2}. 
$$
Therefore
\begin{equation}
\label{e3.18}
M^*(F^n_g)(x)=- \frac{\mbox{\rm Tr}\, [P_n(-A)^{-2\beta}] }{n^{-1}+2|(-A)^{-\beta} x|^2}+\frac{4 |(-A)^{-2\beta}P_nx|^2}{(n^{-1}+2|(-A)^{-\beta} x|^2)^2} +\frac{\sum_{h=1}^n\alpha_h^{-\beta}x_hv_h}{n^{-1}+2|(-A)^{-\beta} x|^2}.
\end{equation}
\begin{Proposition}
\label{p4.2e}
For any $r>0$ there exists   the  {\em surface measure} $\sigma_r$  on $B_{\sqrt r}$ determined by   $\nu$.

\end{Proposition}
\begin{proof}
Since obviously
$$
\lim_{n\to\infty}F^n_g=F_g\quad\mbox{\rm in}\;L^2(H,\nu;H)
$$
and $M^*$ is a closed operator, to prove that $F_g\in D(M^*)$  it is enough to show that there is a constant  $K>0$ such that
\begin{equation}
\label{e5.1}
\|M^*(F^n_g)\|_{L^2(H,\nu)}\le K,\quad\forall\;n\in\N.
\end{equation}
To prove \eqref{e5.1}  we shall proceed in two steps.\medskip

{\it Step 1}. We show that $|(-A)^{-\beta}x|^{-2}\in L^{k+1}(H,\nu) $ for all $k\in\N$. \medskip

Let $d\in \N$ (to be chosen later) and let $P$ be the orthogonal projector  on span $(e_1,...,e_d)$. Write
$$
\frac1{|(-A)^{-\beta}x|^2}\le \frac1{|(-A)^{-\beta}Px|^2}\le \frac{\alpha^\beta_d}{|Px|^2}.
$$  
Then it is enough to show that
$$
\frac1{|Px|^2}\in L^k(H,\nu)
$$
 We apply It\^o's formula to $\varphi_\epsilon(X(t))$ where $X$ is the solution to \eqref{e4.1} and 
$$
\varphi_\epsilon(x)=\frac1{(\epsilon+|Px|^2)^k},\quad x\in H,
$$
so that
$$
(D\varphi_\epsilon(x),u)=-\frac{2k\langle Px,Pu\rangle }{(\epsilon+|Px|^2)^{k+1}}\quad x,u\in H
$$
and
$$
D^2_x\varphi_\epsilon(x)(u,v)=-2k\frac{(Pu,Pv)}{(\epsilon+|Px|^2)^{k+1}}+4k(k+1)\frac{(Px,Pu)(Px,Pv)}{(\epsilon+|Px|^2)^{k+2}}\quad x,u,v\in H.
$$
Therefore
$$
\mbox{\rm Tr}\;[D^2_x\varphi_\epsilon(x)]=
- \frac{2kd}{(\epsilon+|Px|^2)^{k+1}}+4k(k+1)\frac{|Px|^2}{(\epsilon+|Px|^2)^{k+2}}
$$
and we have
 \begin{equation}
\label{e3.22}
\begin{array}{l}
\ds \varphi_\epsilon(X(t))-\varphi_\epsilon(x) -2k
\frac{|(-A)^{1/2}PX(t)|^2}{(\epsilon+|PX(t)|^2)^{k+1}}\;dt =-2k\,\frac{\langle Pp(X(t)),PX(t)\rangle}{(\epsilon+|PX(t)|^2)^{k+1}}\;dt\\
\\
\ds - \frac{kd}{(\epsilon+|PX(t)|^2)^{k+1}}\,dt+2k(k+1)\frac{|P(X(t)|^2}{(\epsilon+|P(X(t)|^2)^{k+2}}\,dt+ dG_t,
\end{array} 
\end{equation}
where $G_t$ is a martingale.

We assume now that $X$ is a stationary process with law $\nu$. We integrate between $0$ and $1$ the identity \eqref{e3.22} and we take the expectation.
We get
\begin{equation}
\label{e3.23}
\begin{array}{l}
\ds   kd\int_H \frac{1}{(\epsilon+|Px|^2)^{k+1}}\,\nu(dx)=
2k\int_H
\frac{|(-A)^{1/2}Px|^2}{(\epsilon+|Px|^2)^{k+1}}\;\nu(dx)\\
\\
\ds -2k\int_H\frac{\langle Pp(x),Px\rangle}{(\epsilon+|Px|^2)^{k+1}}\;\nu(dx) +2k(k+1)\int_H\frac{|Px|^2}{(\epsilon+|Px|^2)^{k+2}}\,\nu(dx)\\
\\
=:I_1+I_2+I_3.
\end{array} 
\end{equation}
Let us estimate $I_1$.
 Since
$$
|(-A)^{1/2}Px|^2\le \alpha_d |Px|^2\le \alpha_d(\epsilon+|Px|^2),
$$
given $\rho>0$ and using H\"older's and Young's inequalities, we see that there is a constant 
$C_{\rho,k,d}>0$ such that
\begin{equation}
\label{e9d}
\begin{array}{lll}
| I_1|&\le& \ds 2k
\alpha_d \int_H
\frac{1}{(\epsilon+|Px|^2)^{k}}\;\nu(dx) \le 2k\alpha_d\left( \int_H
\frac{1}{(\epsilon+|Px|^2)^{k+1}}\;\nu(dx)  \right)^{k/(k+1)}\\
\\
&\le& \ds C_{\rho,k,d}+\rho\int_H
\frac{1}{(\epsilon+|Px|^2)^{k+1}}\;\nu(dx).
\end{array} 
\end{equation}
Concerning $I_2$, noticing that
$$
\langle Pp(x),Px\rangle\le |Pp(x)|\,|Px|\le  |Pp(x)|\,
(\epsilon+|Px|^2)^{1/2},
$$
using H\"older's and Young's inequalities and recalling \eqref{e1.4g}, we see that there is a constant 
$C_{\rho,k}>0$ such that
\begin{equation}
\label{e10d}
\begin{array}{l}
\ds |I_2|\le 2k\int_H\frac{|Pp(x)|}{(\epsilon+|Px|^2)^{k+1/2}}\;\nu(dx)\\
\\
\ds\le\left(\int_H|Pp(x)|^{2k+2} d\nu  \right)^{\frac{1}{2k+2}}\;\left(\int_H\frac{1}{(\epsilon+|Px|^2)^{k+1}}\;\nu(dx)   \right)^{\frac{2k+1}{2k+2}} \\
\\
\ds\le C_{\rho,k}+\rho\int_H\frac{1}{(\epsilon+|Px|^2)^{k+1}}\;\nu(dx)
\end{array} 
\end{equation}
Finally, obviously
\begin{equation}
\label{e11d}
|I_3| \le 2k(k+1)
\int_H\frac{1}{(\epsilon+|Px|^2)^{k+1}}\;\nu(dx).
\end{equation}
By \eqref{e9d},  \eqref{e10d} and  \eqref{e11d} it follows that
$$
\begin{array}{l}
\ds kd\int_H\frac{1}{(\epsilon+|Px|^2)^{k+1}}\;\nu(dx)\le C_{\rho,k,d}+ C_{\rho,k}\\
\\
\ds+(2k(k+1)+2\rho)
\int_H\frac{1}{(\epsilon+|Px|^2)^{k+1}}\;\nu(dx).
\end{array}
$$
Now we choose $d$ and $\rho$ such that
$$
d>2(k+1),\quad kd>2k(k+1)+2\rho
$$ 
and conclude that there exists $M>0$ such that
$$
\int_H\frac{1}{(\epsilon+|Px|^2)^{k+1}}\;\nu(dx)\le M.
$$
The conclusion follows letting $\epsilon\to 0$.\medskip

{\it Step 2}. Conclusion.\medskip

Set $\||(-A)^{-\beta}x|^{-2}\|_{L^{k}(H,\nu)}=D_k$
and 
write $M^*(F^n_g)=J^n_1+J^n_2,$ where
$$
J^n_1=- \frac{\mbox{\rm Tr}\, [(-A)^{-2\beta}P_n] }{n^{-1}+2|(-A)^{-\beta} x|^2}+\frac{4 |(-A)^{-2\beta}x|^2}{|(n^{-1}+(-A)^{-\beta} x|^2)^2},\quad J_2=\frac{\sum_{h=1}^n\alpha_h^{-\beta}x_hv_h}{n^{-1}2|(-A)^{-\beta} x|^2}.
$$
Since
$$
|(-A)^{-2\beta}x|^2\le \alpha_1^{-2\beta}|(-A)^{-\beta}x|^2,
$$
we have
\begin{equation}
\label{e4.10e}
J^n_1\le \frac{\mbox{\rm Tr}\;[(-A)^{-2\beta}]+2\alpha_1^{-2\beta}}{|(-A)^{-\beta}x|^2}
\end{equation}
Therefore by Step 1,
\begin{equation}
\label{e4.11e}
\|J^n_1\|_{L^2(H,\nu)}\le D_1(\mbox{\rm Tr}\;[(-A)^{-2\beta}]+2\alpha_1^{-2\beta}).
\end{equation}
Concerning $J_2^n$ we first notice that, thanks to \eqref{e3.5e} and \eqref{e1.4g} there is $L>0$ such that
$$
\begin{array}{l}
\ds\left\|\sum_{h=1}^\infty\alpha_h^{-\beta}x_hv_h\right\|_{L^4(H,\nu)}\le \sum_{h=1}^\infty \alpha_h^{-\beta} \|x_h\|_{L^8(H,\nu)}\,\|v_h\|_{L^8(H,\nu)}\\\\
\ds\le c_{8/7}\sum_{h=1}^\infty\alpha_h^{-\beta}\|x\|_{L^8(H,\nu)}\le L.
\end{array}
$$
Now  by   H\"older's inequality it follows that
\begin{equation}
\label{e3.21}
\|J^n_2\|_{L^2(H,\nu)}\le L  \||(-A)^{-1}x|^{-2}\| _{L^8(H,\nu)}\le LD_8.
\end{equation}
The conclusion follows.

\end{proof}

 \subsection{Surface measure of a  half-space}
 
   Let  $n=1,2,3$ and $g(x)=\langle x,b\rangle,$  with $b\in H$. Then for any $r\in \R$, $\{g=r\}$ is n half-space which we denote by $S_{b,r}$.   
   \begin{Proposition}
\label{p4.3e}
Assume that
\begin{equation}
\label{e3.19}
\sum_{j=1}^\infty\alpha_j^{-1} a_j<\infty.
\end{equation}
Then there exists   the  {\em surface measure} $\sigma_r$  on $S_{a,r}$ determined by   $\nu$.

\end{Proposition}
\begin{proof}
 We  have $Mg=(-A)^{-\beta}b$ and so
 $$
 F_g=\frac{(-A)^{-\beta}b}{|(-A)^{-\beta}b|^2}.
 $$
 By \eqref{e3.20e} it follows that
 $$
 M^*(F_g)=\ \frac{1}{|(-A)^{-\beta}b|^2}\,\sum_{j=1}^\infty\alpha_j^{-\beta}v_j(x) b_j.
 $$
 Therefore
   $$
 \|M^*(F_g)\|_{L^2(H,\nu)} \le 
 \frac{C_2}{|(-A)^{-1}b|^2}\,\sum_{j=1}^\infty\alpha_j^{-1} b_j.
$$
The conclusion follows using a similar approximation argument as above.

\end{proof}\bigskip

\subsection*{Acknowledgement}  G. Da Prato is partially supported by GNAMPA from INDAM. 

\noindent A. Debussche is partially supported by the French government thanks to the
ANR program Stosymap. He also benefits from the support of the French government ``Investissements d'Avenir" program ANR-11-LABX-0020-01.

 Giuseppe Da Prato, Scuola Normale Superiore, 56126, Pisa, 
Italy.   e-mail:   giuseppe. daprato@sns.it \medskip

Arnaud Debussche, IRMAR and \'Ecole Normale Sup\'erieure de Rennes, Campus de Ker Lann, 37170 Bruz, France. e-mail:arnaud.debussche@ens-rennes.fr


\begin{thebibliography}{99}
 
   

 

 


 \bibitem[AiMa88]{AiMa88}
  H. Airault and  P. Malliavin,    {\it Int\'egration g\'eom\'etrique sur l'espace de Wiener},  Bull. Sci. Math. {\bf 112},  3--52, 1988. 
  
  \bibitem[BoMa]{BoMa} V. I. Bogachev and I.I. Malofeev, {\it On surface masures generated by differentiable measures}, preprint.
  
   
 
 
\bibitem[CaCh14]{CaCh14} R. Catellier and K. Chouk, {\it Paracontrolled Distributions and the 3-dimensional Stochastic Quantization Equation}, ArXiv:1310.6869.

\bibitem[Ce01]{Ce01}  S. Cerrai,   {\it Second order PDE's in finite
and infinite dimensions. A probabilistic approach}, Lecture Notes in Mathematics,     
{\bf 1762}, Springer-Verlag, 2001.


 \bibitem[Da04]{Da04} G. Da Prato, {\it Kolmogorov equations for stochastic PDEs}, Birk\"auser 2004.
 
\bibitem[DaDe03]{DaDe-03} G. Da Prato and A. Debussche, {\it Strong solutions to the stochastic quantization equations}, Ann. Probab. 31, no. 4, 1900--1916, 2003.
  
  
 

 \bibitem[DaDe14]{DaDe14} G. Da Prato and A. Debussche, {\it Estimate for $P_tD$ for the stochastic Burgers equation}, Ann. Inst. Henri Poincar\'e Probab. Stat. (to appear), arXiv:1412.7426, 2014.
 
 \bibitem[DaDe15]{DaDe15}   G. Da Prato and A. Debussche,  {\it Existence of the Fomin derivative of the invariant measure  of a stochastic reaction--diffusion equation},   RIM workshop 2014 (to appear), arXiv:1193405, 2015.
 
 
  \bibitem[DaLuTu14]{DaLuTu14} G. Da Prato, A. Lunardi  and L. Tubaro,  {\it Surface measures in infinite dimensions}, Rend. Lincei Math. Appl. {\bf 25}, 309--330, 2014.
 
  \bibitem[DaLuTu]{DaLuTu}   G. Da Prato, A. Lunardi and L. Tubaro,  {\it Malliavin Calculus for non--Gaussian measures and surface measures in Hilbert spaces}, in preparation.

  
 \bibitem[FePr92]{FePr92}      D. Feyel and   A. de La Pradelle,  {\it Hausdorff measures on the Wiener space}, Pot. Analysis, {\bf 1},177--189, 1992.    
 
 
 
\bibitem[Hai14]{Hai14} M. Hairer, {\it A theory of regularity structures}, Invent. Math. 198, no. 2, pp. 269-504, 2014.
  



\bibitem[Ma97]{Ma97}    P. Malliavin,  \textit{Stochastic analysis},   Springer-Verlag, Berlin, 1997.




 


 \bibitem[Pu98]{Pu98} O. V. Pugachev, {\it Surface measures in infinite-dimensional spaces}, Mat. Zametki 63 (1998), no. 1, 106--114; translation in {\it Mathematical Notes}, {\bf 63}, no.1-2, 94--101, 1998.

 

\end{thebibliography}
\end{document}